\documentclass{elsarticle}
\usepackage{amsfonts,amssymb,amsmath,amsthm}
\usepackage{url}
\usepackage{enumerate}
\usepackage{enumitem}
\usepackage[numbers]{natbib}
\usepackage{hyperref}
\usepackage[margin=3.5cm]{geometry}

\newtheorem{thrm}{Theorem}[section]
\newtheorem{lem}[thrm]{Lemma}
\newtheorem{prop}[thrm]{Proposition}
\newtheorem{cor}[thrm]{Corollary}
\theoremstyle{definition}
\newtheorem{definition}[thrm]{Definition}
\newtheorem{remark}[thrm]{Remark}
\numberwithin{equation}{section}

\DeclareMathOperator{\Ker}{Ker}

\begin{document}

\title{Products of Ideals and Jet Schemes}
\tnotetext[t1]{Gleb Pogudin's current address: Courant Institute of Mathematical Sciences, New York University, USA.}

\author[gleb]{Gleb Pogudin}
\ead{pogudin.gleb@gmail.com}
\address[gleb]{
Institute for Algebra\\
Johannes Kepler University\\
4040, Linz, Austria}

\begin{abstract}
In the present paper, we give a full description of the jet schemes of the polynomial ideal $\left( x_1\cdots x_n \right) \in k[x_1, \ldots, x_n]$ over a field of zero characteristic.
We use this description to answer questions about  products and intersections of ideals emerged recently in  algorithmic studies of algebraic differential equations.
\end{abstract}

\begin{keyword}
jets schemes \sep algebraic differential equations \sep differential polynomials

\MSC[2010]12H20 \sep 12H05 \sep 14B10
\end{keyword}

\maketitle

\section{Introduction}
Properties of ideals in rings of differential polynomials are essential for the algorithmic study of algebraic differential equations
\cite{GustavsonKondratievaOvchinnikov,OvchinnikovPogudinVo}.
The following fact is an important part of the recent first bound for the effective differential elimination (\cite{OvchinnikovPogudinVo}). 
If $h_1, \ldots, h_n$ are natural numbers, then there exists $d$ such that for every ideals $I_1, \ldots, I_n$ in a commutative associative differential algebra over a field of zero characteristic
\begin{equation}\label{eq:intro_incl}
	\left( I_1^{(h_1)} \cdots I_n^{(h_n)} \right)^d \subset \left( I_1 \cdots I_n\right)^{(h_1 + \ldots + h_n)},
\end{equation}
where, for an ideal $I$, we denote the ideal generated by derivatives of elements of $I$ of order at most $h$ by $I^{(h)}$.
Such an inclusion allows us to reduce a problem about an arbitrary ideal to the problem about an ideal with additional useful properties (for example, to prime ideals in \cite{OvchinnikovPogudinVo}).
This kind of reduction is expected to have many potential applications to the algorithmic problems about algebraic differential equations.
The value of such $d$ was not needed for obtaining the bound in~\cite{OvchinnikovPogudinVo}, but we expect that it can be used for refining this and related bounds similarly to the way the Noether exponent was used in  \cite[Lemma 3.1]{GustavsonKondratievaOvchinnikov} and \cite[Lemma 5.5]{OvchinnikovPogudinVo}.
The problem of determining the number $d$ in \eqref{eq:intro_incl} as well as many other questions about products of ideals in differential algebras can be reduced to questions about the jet ideal of the polynomial ideal generated by $x_1 \cdots x_n \in k[x_1, \ldots, x_n]$ (see \cite[Lemma 6.2]{OvchinnikovPogudinVo}).

Recently, jet schemes were successfully applied to study singularities of algebraic varieties (for example, \cite{Mustata}).
In  \cite{GowardSmith,Yeun}, jet ideals of the ideal generated by $x_1 \cdots x_n$ were studied, minimal primes and their multiplicities were found.
In the context of differential algebra, jet schemes and related objects proved to be one of main tools in differential algebraic geometry (see \cite{MoosaScanlon,LeonSanchez}).

In the present paper, we generalize results of \cite{GowardSmith,Yeun} giving a full description of jet schemes of the ideal generated by $x_1 \cdots x_n$ (Section~\ref{sec:jets}), where an unexpected connection to Shubert calculus occurs (see Proposition~\ref{prop:nilpotent_nonzero}).
Then we use the obtained information in order to solve the original problem about differential algebraic equations, namely find the minimal possible number $d$ in~\eqref{eq:intro_incl} (see Corollary~\ref{cor:product_of_ideals} and Corollary~\ref{cor:intersection}).

We also note that our problem is connected to a classical membership problem for the differential ideal generated by $x_1 \ldots x_n$ (see \cite{Levi,HillmanEtAl,HillmanEtAl2}).
In particular, if a differential polynomial belongs to any of the jet ideals of the polynomial ideal generated by $x_1 \cdots x_n$, then it belongs to the corresponding differential ideal, and if a differential polynomial does not belong to any of these jet ideals, it does not belong to the differential ideal.
Our structural results give an effective way for deciding if any of these situations takes place (see Remark~\ref{rem:levi}).


\section{Preliminaries}\label{sec:preliminaries}

\subsection{Differential algebra}

	Throughout the paper, all fields are assumed to be of characteristic zero.
    Unless otherwise stated, all algebras are commutative, associative, and with unity.
	
    Let $R$ be a ring.
    A map $D\colon R \to R$ satisfying $D(a + b) = D(a) + D(b)$ and $D(ab) = aD(b) + D(a)b$ for all $a, b \in R$ is called a \textit{derivation}.
	A \textit{differential ring} $R$ is a ring with a specified derivation.
    In this case, we will denote $D(x)$ by $x^{\prime}$ and $D^n(x)$ by $x^{(n)}$.
	A differential ring that is a field will be called a \textit{differential field}.
	
    A differential ring $A$ is said to be a \textit{differential $K$-algebra} over a differential field $K$ if $A$ is a $K$-algebra and the restriction of the derivation of $A$ on $K$ coincides with the derivation on $K$.
	An ideal $I$ of a differential ring $R$ is said to be a \textit{differential ideal} if $a^{\prime} \in I$ for all $a \in I$.
	The differential ideal generated by $a_1, \ldots, a_n \in I$ will be denoted by $[a_1, \ldots, a_n]$.
    	
    Let $A$ be a differential $K$-algebra.
	We consider the polynomial ring \[A[x, x^{\prime}, \ldots, x^{(n)}, \ldots],\] where $x, x^{\prime}, x^{\prime\prime}, \ldots$ are algebraically independent variables.
	Extending the derivation from $A$ to $A[x, x^{\prime}, \ldots]$ by $D(x^{(n)}) = x^{(n + 1)}$ we obtain a differential algebra.
	This algebra is called the \textit{algebra of differential polynomials} in $x$ over $A$ and we denote it by $A\{ x\}$.
	Iterating this construction, we define the algebra of differential polynomials in variables $x_1, \ldots, x_n$ over $A$ and denote it by $A\{x_1, \ldots, x_n\}$.

\subsection{Jet Schemes}

	For general scheme-theoretic treatment, we refer reader to \cite{Mustata} (the same object is called the arc space in \cite{MoosaScanlon}).
    We will need only coordinate description of the jet ideal of the polynomial ideal.
    
		Let $I = (g_1, \ldots, g_s)$ be an ideal in the polynomial ring $k[x_1, \ldots, x_n]$ and $m \in \mathbb{Z}_{\geqslant 0} \cup \{ \infty\}$.
        By $J_m(I)$ we will denote the $m$-th jet ideal of $I$ which is an ideal in the polynomial ring
        $$
        R_m = k\left[ x_1^{(0)}, x_1^{(1)}, \ldots, x_1^{(m)}, x_2^{(0)}, \ldots, x_n^{(m)}  \right]
        $$
        generated by coefficients of the following truncated power series
        $$
        g_i\left( x_1^{(0)} + x_1^{(1)}t + \ldots + x_1^{(m)}t^m, \ldots, x_n^{(0)} + x_n^{(1)}t + \ldots + x_n^{(m)}t^m \right) \in k[[t]] / (t^{m + 1}),
        $$
        where $i = 1, \ldots, s$.
    
    For example, the $0$-th jet ideal of $I$ coincides with $I$ after renaming $x_i^{(0)}$ into $x_i$ for all $i$.
    
    \begin{remark}\label{rem:jet_to_diff}
    Let us consider the infinite jet ring $R_{\infty}$ over a differential field $K$.
    It can be endowed with the structure of differential $K$-algebra by $\left( x_i^{(j)}\right)^{\prime} = (j + 1)x_i^{(j + 1)}$.
    Moreover, the differential algebra obtained by this procedure is differentially isomorphic to the algebra of differential polynomials $K\{ y_1, \ldots, y_n \}$ via $\varphi\left( x_i^{(j)} \right) = \frac{y_i^{(j)}}{j!}$.
    Then if coefficients of $g_1, \ldots, g_s$ are constants in $K$, the infinite jet ideal $J_{\infty}(I)$ corresponds to a differential ideal generated by $g_1(y_1, \ldots, y_n), \ldots, g_s(y_1, \ldots, y_n)$ in $K\{ y_1, \ldots, y_n\}$.
    This correspondence will allow us to apply results about jet schemes to problems about algebraic differential equations.
    \end{remark}


\section{Facts about symmetric polynomials}\label{sec:symmetric}

	In this section we collect facts about the ideal generated by symmetric polynomials with zero constant term, which plays a central role in Shubert calculus (see~\cite{Manivel, Monk}).
    We also introduce algebras $A_{\lambda}(k)$ (Definition~\ref{def:main}) which locally describe the jet schemes of the ideal $(x_1\cdots x_n)$ (Section~\ref{sec:jets}).
    
	Consider the algebra $k[z_1, \ldots, z_\ell]$ of polynomials over $k$.
    By $\mathbf{z}$ we denote the tuple $(z_1, \ldots, z_\ell)$.
    By $h_d(\mathbf{z})$ we denote \textit{the complete symmetric polynomial of degree $d$} which is a sum of all monomials of degree $d$ in $z_1, \ldots, z_\ell$:
    $$
    h_d(\mathbf{z}) = \sum\limits_{ i_1 \leqslant i_2 \leqslant \ldots \leqslant i_d } z_{i_1} z_{i_2} \cdots z_{i_d}.
    $$
	
    By $\mathbf{z}_k$ we denote the tuple $(z_k, \ldots, z_\ell)$.
    For example, $\mathbf{z}_1 = \mathbf{z}$, $\mathbf{z}_\ell = (z_\ell)$, and $h_d(\mathbf{z}_\ell) = z_\ell^d$.
    By $I_S \subset k[z_1, \ldots, z_\ell]$ we denote the ideal generated by all symmetric polynomials with respect to $z_1, \ldots, z_\ell$ with zero constant term.
    We will use the following lemma which follows from Proposition~2.1 and Remark~2.1 from \cite{MoraSala}.
    
    \begin{lem}\label{lem:grobner_basis}
		Polynomials $h_1(\mathbf{z}_1), h_{2}(\mathbf{z}_{2}), \ldots, h_{\ell}(\mathbf{z}_\ell)$ constitute a Gr\"obner basis of $I_S$ with respect to lexicographical ordering $z_1 > z_{2} > \ldots > z_\ell$.
        
        Also, for every $j \geqslant i$, the polynomial $h_{j}(\mathbf{z}_i)$ lies in $I_S$.
	\end{lem}
    
    Throughout the rest of the section, we will refer to the basis from Lemma~\ref{lem:grobner_basis} as the basis of $I_S$.
    
    \begin{lem}\label{lem:symmetric_reduction}
		Let $p(\mathbf{z}) \in k[z_1, \ldots, z_\ell]$, and $d \geqslant 1$.
        By $\tilde{p}$ we denote the reduction of $p$ with respect to the Gr\"obner basis of $I_S$ with respect to lexicographic ordering with $z_1 > z_2 > \ldots > z_\ell$.
        Then
        $$p(z_1, \ldots, z_s, z_{s + 1}, \ldots, z_\ell) = p(z_1, \ldots, z_{s + 1}, z_{s}, \ldots, z_\ell)$$
        implies 
        $$\tilde{p}(z_1, \ldots, z_s, z_{s + 1}, \ldots, z_\ell) = \tilde{p}(z_1, \ldots, z_{s + 1}, z_{s}, \ldots, z_\ell).$$
	\end{lem}

	\begin{proof}
    	By $g_i$ we denote the $i$-th element of the Gr\"obner basis of $I_S$.
        Lemma~\ref{lem:grobner_basis} implies that there are $\ell$ of them, and $g_i = h_{i}(\mathbf{z}_i)$.
		We will prove that if $p$ is not reduced with respect to $g_1, \ldots, g_\ell$, then there exists $q(\mathbf{z})$ such that the leading monomial of $q$ is lower then the leading monomial of $p$, $q$ is also symmetric with respect to $z_s$ and $z_{s + 1}$, and $p - q \in I_S$.
        We can not apply this procedure infinitely many times, so at some point we will obtain $q$ such that $q$ is symmetric with respect to $z_s$ and $z_{s + 1}$ and reduced with respect to $g_1, \ldots, g_\ell$.
        The uniqueness of the reduction with respect to a Gr\"obner basis implies that $q = \tilde{p}$.
        
        So, assume that $p$ is not reduced with respect to $I_S$.
        Since the reduction with respect to a Gr\"obner basis 
        is always equal to the lead-reduction, 
        the leading monomial of $p$ is not reduced with respect to $g_i$ for some $i$.
        Let the leading monomial of $p$ be $M_1 = z_1^{d_1}\cdots z_\ell^{d_\ell}$.
        Due to the symmetry, the monomial $M_2 = z_1^{d_1}\cdots z_s^{d_{s + 1}} z_{s + 1}^{d_s} \cdots z_\ell^{d_\ell}$ occurs in $p$ with the same coefficient, say $c$.
        Since $M_1$ is greater or equal then $M_2$, $d_s \geqslant d_{s + 1}$.
        We consider three cases
        \begin{enumerate}[noitemsep, topsep=0pt]
			\item $i \neq s$ and $i \neq s + 1$.
            Let us consider 
            \[
            q = 
            \begin{cases}
            p - c\frac{M_1 + M_2}{z_i^{i}} g_i, \text{ if } d_s > d_{s + 1},\\
            p - c\frac{M_1}{z_i^{i}} g_i, \text{ if } d_s = d_{s + 1}.
            \end{cases}
            \]
            The terms $cM_1$ cancel, so the leading monomial of $q$ is less then $M_1$.
            Moreover, $q$ is still symmetric with respect to $z_s$ and $z_{s + 1}$.
            
            \item $i = s$.
            Since $d_s \geqslant s$, Lemma~\ref{lem:grobner_basis} implies that $h_{d_s}(\mathbf{z}_s), h_{d_s + i}(\mathbf{z}_{s + 1}) \in I_S$ for every $i > 0$.
            By $N$ we denote $\frac{M_1}{z_s^{d_s} z_{s + 1}^{d_{s + 1}}}$.
            Consider
            $$
            q = p - cN \left( z_{s + 1}^{d_{s + 1}} h_{d_s}(\mathbf{z}_s) - \sum\limits_{i = 1}^{d_{s + 1}} z_s^{d_{s + 1} - i} h_{d_{s} + i}(\mathbf{z}_{s + 1}) \right).
            $$
            We denote $z_{s + 1}^{d_{s + 1}} h_{d_s}(\mathbf{z}_s)$ and $\sum\limits_{i = 1}^{d_{s + 1}} z_s^{d_{s + 1} - i} h_{d_{s} + i}(\mathbf{z}_{s + 1})$ by $f_1$ and $f_2$, respectively.
            Since $d_s \geqslant d_{s + 1}$, the leading monomial of $f_1 - f_2$ is $z_{s}^{d_s} z_{s + 1}^{d_{s + 1}}$.
            Hence, the leading monomial of $q$ is less than the leading monomial of $p$.
            Since $N$ does not involve neither $z_s$, nor $z_{s + 1}$, it is sufficient to prove that $f_1 - f_2$ is symmetric with respect to $z_{s}$ and $z_{s + 1}$.
            The polynomial $f_1$ is a sum of all monomials $S$ in $\mathbf{z}_s$ of the total degree $d_s + d_{s + 1}$ with $\deg_{z_{s + 1}} S \geqslant d_{s + 1}$.
            The polynomial $f_2$ is a sum of all monomials $S$ in $\mathbf{z}_s$ of the total degree $d_s + d_{s + 1}$ with $\deg_{z_s} S < d_{s + 1}$.
            Let $S_1 = z_1^{e_1} \cdots z_\ell^{e_\ell}$ be an arbitrary monomial of the total degree $d_s + d_{s + 1}$, 
            and $S_2 = z_1^{e_1}\cdots z_s^{e_{s + 1}} z_{s + 1}^{e_s} \cdots z_\ell^{e_\ell}$ is the monomial symmetric to $S_1$ with respect to $z_s$ and $z_{s + 1}$.
            We also assume that $e_s > e_{s + 1}$.
            It is sufficient to prove that $S_1$ and $S_2$ occur in $f_1 - f_2$ with the same coefficient.
            There are three cases:
            \begin{enumerate}
				\item $e_s > e_{s + 1} \geqslant d_{s + 1}$. Then $S_1$ and $S_2$ occur in $f_1$, but none of them occurs in $f_2$.
                \item $d_{s + 1} > e_s > e_{s + 1}$. Then $S_1$ and $S_2$ occur in $f_2$, but none of them occurs in $f_1$.
                \item $e_s \geqslant d_{s + 1} > e_{s + 1}$. Then $S_2$ occurs both in $f_1$ and $f_2$, so $f_1 - f_2$ does not involve $S_2$.
                Also, $S_1$ does not occur in $f_1$ and $f_2$.
			\end{enumerate}
            
            \item $i = s + 1$.
            Since $d_s \geqslant d_{s + 1}$, $p$ is also not reduced with respect to $g_s$, so this case can be reduced to the previous.
		\end{enumerate}
	\end{proof}

	\begin{definition}\label{def:main}
    	Let us remind that by $I_S$ we denote the ideal of $k[z_1, \ldots, z_\ell]$ generated by symmetric polynomials with respect to $z_1, \ldots, z_\ell$ without constant term.
        For the rest of the section, we fix a positive integer $n$.
	    For the vector $\lambda = (\lambda_1, \ldots, \lambda_n) \in \mathbb{Z}_{\geqslant 0}^n$, by $\Lambda_i$ we denote the partial sum $\lambda_1 + \ldots + \lambda_{i - 1}$, assuming also $\Lambda_1 = 0$.
	    Assume that $\Lambda_{n + 1} = \lambda_1 + \ldots + \lambda_n = \ell$.
        We will say that a polynomial $p \in k[z_1, \ldots, z_\ell]$ is $\lambda$-symmetric if it is symmetric with respect to $z_{\Lambda_i + 1}, \ldots, z_{\Lambda_{i + 1}}$ for every $1 \leqslant i \leqslant n$.
        The notion of $\lambda$-symmetric polynomials generalizes the notion of block-symmetric polynomials~\cite{BlockSymmetric}.
	\end{definition}
    
    For Lemmas~\ref{lem:ideal} and~\ref{lem:fundamental}, we fix some $\lambda = (\lambda_1, \ldots, \lambda_n) \in \mathbb{Z}_{\geqslant 0}^n$ such that $\lambda_1 + \ldots + \lambda_n = \ell$.
    We denote the elementary symmetric polynomial of degree $d$ in $t_1, \ldots, t_r$ 
    by $\sigma_d(t_1, \ldots, t_r)$
    for any variables $t_1, \ldots, t_r$ and $1 \leqslant d \leqslant r$.
       
    \begin{lem}\label{lem:ideal}
      Let $p \in k[z_1, \ldots, z_\ell]$ be a $\lambda$-symmetric polynomial belonging to $I_S$.
      Then there exist $\lambda$-symmetric polynomials $a_1, \ldots, a_\ell \in k[z_1, \ldots, z_\ell]$
      such that
      \[
      p = a_1\sigma_1(z_1, \ldots, z_{\ell}) + \ldots + a_\ell\sigma_\ell(z_1, \ldots, z_{\ell}).
      \]
    \end{lem}
    
    \begin{proof}
      Since $p \in I_S$, there exist $b_1, \ldots, b_\ell \in k[z_1, \ldots, z_{\ell}]$ such that 
      \begin{equation}\label{eq:ideal_raw}
      p = b_1\sigma_1(z_1, \ldots, z_{\ell}) + \ldots + b_\ell\sigma_\ell(z_1, \ldots, z_{\ell}).
      \end{equation}
      Let $G := S_{\lambda_1} \times \ldots \times S_{\lambda_n} \subset S_\ell$ be a subgroup in the group of all permutations of $z_1, \ldots, z_\ell$.
      We consider a linear operator $\operatorname{Sym}_\lambda$ on $k[z_1, \ldots, z_m]$ defined by 
      \[
      \operatorname{Sym}_\lambda (q) := \frac{1}{|G|} \sum\limits_{\sigma \in G} \sigma(q), \text{ for } q \in k[z_1, \ldots, z_\ell].
      \]
      Then we have $\operatorname{Sym}_\lambda (a b) = a \operatorname{Sym}_\lambda (b)$ for every $\lambda$-symmetric $a$ and every $b \in k[z_1, \ldots, z_\ell]$. Moreover $\operatorname{Sym}_\lambda (b)$ is $\lambda$-symmetric for every $b \in k[z_1, \ldots, z_\ell]$.
      Applying $\operatorname{Sym}_\lambda$ to both sides of~\eqref{eq:ideal_raw}, we obtain
      \[
      p = \operatorname{Sym}_\lambda(b_1)\sigma_1(z_1, \ldots, z_{\ell}) + \ldots + \operatorname{Sym}_\lambda(b_\ell)\sigma_\ell(z_1, \ldots, z_{\ell}).
      \]
      Since every polynomial in the image of $\operatorname{Sym}_\lambda$ is $\lambda$-symmetric, we set $a_i := \operatorname{Sym}_\lambda(b_i)$ for every $1 \leqslant i \leqslant \ell$.
    \end{proof}
    
    \begin{lem}\label{lem:fundamental}
      Let $p \in k[z_1, \ldots, z_\ell]$ be a $\lambda$-symmetric polynomial.
      Then $p$ belongs to the subalgebra generated by
      \[
      \left\{ \sigma_j(z_{\Lambda_i + 1}, \ldots, z_{\Lambda_{i + 1}}) \mid 1 \leqslant i \leqslant n, 1 \leqslant j \leqslant \lambda_i \right\}.
      \]
    \end{lem}
    
    \begin{proof}
      We will prove the statement of the lemma by induction on $n$.
      The base case $n = 1$ is exactly the fundamental theorem on symmetric polynomials.
      Let $n > 1$ and consider a $\lambda$-symmetric polynomial $p \in k[z_1, \ldots, z_\ell]$.
      We write $p$ in the form $\sum\limits_{i = 1}^N a_i m_i$, where $a_i \in k[z_1, \ldots, z_{\Lambda_{n - 1}}]$
      and $m_i$ are distinct monomials in $z_{\Lambda_{n - 1} + 1}, \ldots, z_{\Lambda_n}$.
      Since $p$ is $\lambda$-symmetric, $a_i = a_j$ for every $1 \leqslant i, j \leqslant N$ such that $m_i$ and $m_j$
      can be obtained from each other by permuting the variables.
      Grouping together $m_i$'s with equal $a_i$'s, we obtain a representation $p = \sum\limits_{i = 1}^M b_i s_i$, where
      $b_i \in k[z_1, \ldots, z_{\Lambda_{n - 1}}]$ and $s_i$ is a symmetric polynomial in $z_{\Lambda_{n - 1} + 1}, \ldots, z_{\Lambda_n}$.
      The fundamental theorem on symmetric polynomials imply that each $s_i$ is contained in the subalgebra generated by 
      $\sigma_1(z_{\Lambda_{n - 1} + 1}, \ldots, z_{\Lambda_n})$, $\ldots$, $\sigma_{\lambda_n}(z_{\Lambda_{n - 1} + 1}, \ldots, z_{\Lambda_n})$.
      Since $p$ is $\lambda$-symmetric, $b_1, \ldots, b_M$ are $(\lambda_1, \ldots, \lambda_{n - 1})$-symmetric, so they belong to the subalgebra generated by
      \[
      \left\{ \sigma_j(z_{\Lambda_i + 1}, \ldots, z_{\Lambda_{i + 1}}) \mid 1 \leqslant i \leqslant n - 1, 1 \leqslant j \leqslant \lambda_i \right\}
      \]
      due to the induction hypothesis.
      Thus, the lemma is proved.
    \end{proof}
    
    
    \begin{definition}
        By $A_{\lambda}(k)$ we denote the subalgebra of $k[z_1, \ldots, z_\ell] / I_S$ consisting of images of all $\lambda$-symmetric polynomials.
    \end{definition}
    
    \begin{lem}\label{lem:dimension}
	    For all $\ell \in \mathbb{Z}_{>0}$ and $\lambda \in \mathbb{Z}_{\geqslant 0}^n$ with $\lambda_1 + \ldots + \lambda_n = \ell$,
		$$\dim_k A_{\lambda}(k) = \frac{\ell !}{\prod\limits_{i = 1}^n \lambda_i!}.$$
	\end{lem}

	\begin{proof}
		Let us identify every element of $A_{\lambda}(k)$ with its normal form with respect to the Gr\"obner basis $\{g_1, \ldots, g_\ell\}$ from Lemma~\ref{lem:grobner_basis}.
        Lemma~\ref{lem:symmetric_reduction} implies that $A_{\lambda}(k)$ consists of polynomials $p$ such that $p$ is reduced with respect to $\{ g_1, \ldots, g_\ell \}$, and $p$ is $\lambda$-symmetric.
        By $G$ we denote the subgroup of $S_\ell$ such that for every $\sigma \in G$ and $\Lambda_i + 1 \leqslant j \leqslant \Lambda_{i + 1}$, $\sigma(j)$ satisfies $\Lambda_i + 1 \leqslant \sigma(j) \leqslant \Lambda_{i + 1}$.
        Then $G \cong S_{\lambda_1} \times S_{\lambda_2} \times \ldots \times S_{\lambda_n}$.
        
        By $B$ we denote the set of vectors $\mathbf{d} = (d_1, \ldots, d_\ell) \in \mathbb{Z}_{\geqslant 0}^\ell$ such that $d_{\Lambda_i + 1} \leqslant \Lambda_i$ and $d_{\Lambda_i + 1} \geqslant d_{\Lambda_i + 2} \geqslant \ldots \geqslant d_{\Lambda_{i + 1}}$ for every $i \leqslant n$.
        For every such vector, we associate the polynomial 
        \[
        f_{\mathbf{d}} = \sum\limits_{\sigma \in G} z_{\sigma(1)}^{d_1} z_{\sigma(2)}^{d_2} \cdots z_{\sigma(\ell)}^{d_\ell}.
        \]
        We claim that these polynomials constitute a $k$-basis of $A_{ \lambda }(k)$.
        The leading monomial of $f_{\mathbf{d}}$ is $z_1^{d_1}\cdots z_\ell^{d_\ell}$, so polynomials $f_{\mathbf{d}}$ have distinct leading monomials.
        Hence, they are linearly independent.
        It is sufficient to prove that for every $p \in A_{\lambda}(k)$, the leading monomial of $p$ coincides with the leading monomial of $f_{\mathbf{d}}$ for some $\mathbf{d}$.
        If $M = z_1^{e_1}\cdots z_\ell^{e_\ell}$ is the leading monomial of $p$, then $e_{\Lambda_i + 1} \leqslant \Lambda_i$ for all $i$ since $p$ is reduced with respect to the Gr\"obner basis.
        Assume also that for some $k$ and $i$ such that $\Lambda_i + 1 \leqslant k < \Lambda_{i + 1}$, we have $e_k < e_{k + 1}$.
        Due to the symmetry, the monomial $M\frac{z_k^{e_{k + 1}} z_{k + 1}^{e_k}}{z_{k}^{e_k} z_{k + 1}^{e_{k + 1}}}$, which is larger than $M$, occurs in $p$.
        Hence, the exponent of $M$ lies in $B$.
        
        Thus, we need to compute the cardinality of $B$.
        For every fixed $i$, the number of tuples $(d_{\Lambda_i + 1}, \ldots, d_{\Lambda_{i + 1}})$ such that $\Lambda_i \geqslant d_{\Lambda_i + 1} \geqslant \ldots \geqslant d_{\Lambda_{i + 1}}$ is equal to the number of multisets of size $\Lambda_{i + 1} - \Lambda_i = \lambda_i$ whose elements are chosen from $\{ 0, 1, \ldots, \Lambda_i \}$.
        The number of such multisets is $\binom{\Lambda_i + \lambda_i}{\lambda_i} = \binom{\Lambda_{i + 1}}{\lambda_i}$.
        Hence, the total number of elements in $B$ is
        \[
        \prod\limits_{i = 1}^n \binom{\Lambda_{i + 1}}{\lambda_i} = \frac{\ell !}{\prod\limits_{i = 1}^n \lambda_i! }.\qedhere
        \]
	\end{proof}

  \begin{definition}
  For every $i \leqslant n$, by $A_{\lambda}^i(k)$ we denote the subalgebra of  $A_{\mathbb{\lambda}}(k)$ consisting of images of symmetric polynomials from $k[z_{\Lambda_i + 1}, \ldots, z_{\Lambda_{i + 1}}]$.
	    
  For every nonzero $a \in A_{\lambda}(k)$, by $\nu(a)$ we denote the minimal total degree of a monomial in the reduction of $a$ with respect to the Gr\"obner basis of $I_S$.
  Note that since $I_S$ is homogeneous, $\nu(a)$ does not depend on the ordering of variables.
  We also observe that if $ab \neq 0$, then $\nu(ab) = \nu(a) + \nu(b)$.
  \end{definition}

  The following Propositions~\ref{prop:nilpotent_zero} and~\ref{prop:nilpotent_nonzero} are important for applications of our results to differential algebra (see the proof of Theorem~\ref{th:degree}).

	\begin{prop}\label{prop:nilpotent_zero}
		Let $a \in A_{\lambda}^i(k)$ be an element such that $\nu(a) = r > 0$.
        Then $a^{d} = 0$ for all $d > \frac{\lambda_i}{r}(\ell - \lambda_i)$.
	\end{prop}
    
    \begin{proof}
		Reordering variables if necessary, we can assume that $i = 1$.
        Then $a$ is an image of $p(z_1, \ldots, z_{\lambda_1})$, where $p$ is a symmetric polynomial in $z_1, \ldots, z_{\lambda_1}$ over $k$ without constant term.
        Let $g_1, \ldots, g_\ell$ be the Gr\"obner basis of $I_S$ from Lemma~\ref{lem:grobner_basis}.
        Lemma~\ref{lem:symmetric_reduction} implies, that for every $d$, the reduction of $p(z_1, \ldots, z_{\lambda_1})^d$ with respect to this basis is a polynomial symmetric with respect to $z_1, \ldots, z_{\lambda_1}$ and with respect to all other variables.
        Let us denote this reduction by $q$.
        Since $\deg_{z_1} q < \deg_{z_1} g_1 = 1$, $q$ does not involve $z_1$.
        Symmetry implies that $q$ does not involve also $z_2, \ldots, z_{\lambda_1}$.
        Analogously, for $i > \lambda_1$ we have $\deg_{z_i} q < \deg_{z_{\lambda_1 + 1}} g_{\lambda_1 + 1} = \lambda_1 + 1$.
        Hence, the total degree of $q$ does not exceed $\lambda_1 (\ell - \lambda_1)$.
        On the other hand, $I_S$ is a homogeneous ideal, so the total degree of $q$ can not be less than $\nu(p^d) = dr$.
        Thus, for all $d > \frac{\lambda_1}{r}(\ell - \lambda_1)$, $q$ is zero.
	\end{proof}

	\begin{prop}\label{prop:nilpotent_nonzero}
		Let $a = x_{\Lambda_i + 1} + \ldots + x_{\Lambda_{i + 1}} \in A_{\lambda}^i(k)$.
        Then $a^{\lambda_i(\ell - \lambda_i)} \neq 0$.
	\end{prop}
    
    \begin{proof}
		Our proof will use some facts about Shubert polynomials, in particular, Monk's formula.
        For details and profound account, we refer reader to the original paper \cite{Monk} and the monograph \cite{Manivel}.
        We will briefly recall the most important facts.
        Shubert polynomials are polynomials in $z_1, \ldots, z_\ell$ indexed by permutations.
        By $\mathfrak{S}_{\omega}$ we will denote the Shubert polynomial corresponding to a permutation $\omega \in S_\ell$.
        We will use several properties of these polynomials:
        \begin{enumerate}
            \item The set of all Shubert polynomials in $z_1, \ldots, z_\ell$ constitute a $k$-basis of $k[z_1, \ldots, z_\ell] / I_S$ (see~\cite[Prop. 2.5.3]{Manivel}).
            \item By $s_i$ we denote the transposition $(i, i + 1) \in S_\ell$. Then $\mathfrak{S}_{s_i} = z_1 + \ldots + z_i$ (see~\cite[\S 2.7.1]{Manivel}).
            \item (Monk's formula) For every permutation $\omega \in S_\ell$:
            $$
            \mathfrak{S}_{s_r} \cdot \mathfrak{S}_{\omega} \equiv \sum\limits_{\substack{ j \leqslant r < k \\ l(\omega (j k)) = l(\omega) + 1 } } \mathfrak{S}_{\omega (jk)} \pmod{I_S},
            $$
            where $l(\omega)$ is the number of inversions in $\omega$ (see~\cite[Th. 2.7.1]{Manivel} and \cite[Th. 3]{Monk}).
		\end{enumerate}
        
        Let us return to the proof of the proposition.
        Reordering variables if necessary, we may assume that $i = 1$.
        Using the language of Shubert polynomials, we want to prove that $\mathfrak{S}_{s_{\lambda_1}}^{\lambda_1(\ell - \lambda_1)} = \mathfrak{S}_{s_{\lambda_1}}^{\lambda_1(\ell - \lambda_1)}\mathfrak{S}_{e} \neq 0 \pmod{I_S}$, where $e$ is the identity in $S_\ell$.
        The expression $\mathfrak{S}_{s_{\lambda_1}}^{\lambda_1(\ell - \lambda_1)}$ can be represented as a sum of Shubert polynomials with non negative integer coefficients 
        by applying Monk's formula $\lambda_1(\ell - \lambda_1) - 1$ times.
        Due to Monk's formula, every summand in this representation corresponds to some permutation $\omega$ such that $l(\omega) = \lambda_1(\ell - \lambda_1)$ 
        and $\omega$ can be written as a product of $\lambda_1(\ell - \lambda_1)$ transpositions of the form $(j k)$, where $j \leqslant \lambda_1 < k$.
        Showing that there exists at least one such permutation would imply that there is at least one such summand, so $\mathfrak{S}_{s_{\lambda_1}}^{\lambda_1(\ell - \lambda_1)} \neq 0 \pmod{I_S}$.
        We claim that the following permutation satisfies both properties
        $$
        \omega = 
        \begin{pmatrix}
			1                    & 2                    & \ldots & \lambda_1    & \lambda_1 + 1 & \lambda_1 + 2 & \ldots & \ell            \\
            \ell - \lambda_1 + 1 & \ell - \lambda_1 + 2 & \ldots & \ell         & 1             & 2             & \ldots & \ell - \lambda_1
		\end{pmatrix}.
        $$
        Direct computation shows that the number of inversions is exactly $\lambda_1(\ell - \lambda_1)$.
        The second condition is satisfied, because $\omega$ admits the following representation
        \[
        \omega = \prod\limits_{i = 1}^{\lambda_1} (i \ell) (i \ \ell - 1) \cdots (i \ \lambda_1 + 1).\qedhere
        \]
	\end{proof}
    
    \begin{remark}
		It can be shown, that $\omega$ in the proof above is the only permutation such that $l(\omega) = \lambda_1(\ell - \lambda_1)$ 
        and $\omega$ can be written as a product of $\lambda_1(\ell - \lambda_1)$ transpositions of the form $(j k)$, where $j \leqslant \lambda_1 < k$.
        Then we will obtain the congruence $\mathfrak{S}_{s_{\lambda_1}} \equiv N \mathfrak{S}_{\omega} \pmod{I_S}$, where $N$ is the number of representations of $\omega$ as a product of $\lambda_1(\ell - \lambda_1)$ transpositions of the form $(j k)$, where $j \leqslant \lambda_1 < k$.
        
        For the special case $\lambda = (2, 1, \ldots, 1)$, there is a one-to-one correspondence between the set of such representations and the set of expressions containing $\ell - 2$ pair of parenthesis, which are correctly matched,
        where transpositions of the form $(1 i)$ correspond to open parenthesis, and transpositions of the form $(2 i)$ correspond to close parenthesis (it can be proved using the combinatorial discussion after Th. 3 in \cite{Monk}).
        Thus, in this case $N = C_{\ell - 2}$, the $\ell - 2$-th Catalan number.
        Moreover, since $\omega$ is a Grassmanian permutation, \cite[Prop. 2.6.8]{Manivel} implies that $\mathfrak{S}_{\omega} = z_3^2\cdots z_\ell^2$.
        Hence, we obtain the following curious congruence
        $$
        (z_1 + z_2)^{2(\ell - 2)} \equiv C_{\ell - 2} z_3^2 \cdots z_\ell^2 \pmod{I_S}.
        $$
	\end{remark}

\section{Structure of the jet schemes of $x_1\cdots x_n$}\label{sec:jets}

	The following theorem describes the structure of the $m$-th jet ideal of $(x_1\cdots x_n)$ in terms of algebras $A_{\lambda}(k)$ introduced in Definition~\ref{def:main}.
	For an arbitrary $\lambda = (\lambda_1, \ldots, \lambda_n) \in \mathbb{Z}_{\geqslant 0}^n$ such that $\lambda_1 + \ldots + \lambda_n = m + 1$, by $\mathfrak{p}_{\lambda}$ we denote the prime ideal in $R = k\left[ x_i^{(j)} \mid 1 \leqslant i \leqslant n, 0 \leqslant j \leqslant m \right]$ generated by polynomials $x_{i}^{(j)}$ with $1 \leqslant i \leqslant n$ and $0 \leqslant j < \lambda_i$.
    In \cite[Th. 2.2]{GowardSmith} it is proved that minimal primes of $J_m(x_1\cdots x_n)$ are exactly all ideals of the form $\mathfrak{p}_{\lambda}$.    

	\begin{thrm}\label{th:jet_scheme}
    	Let $n, m$ be positive integers. 
        Consider $R$ and $\mathfrak{p}_{\lambda}$ defined above.
        If we denote the subfield $k\left( x_i^{(j)} \mid 1 \leqslant i \leqslant n, \lambda_i \leqslant j \leqslant m \right)$ of the field of fractions of $R$ by $K_{\lambda}$, then
        $$
        \left( R / J_m\left( x_1 \cdots x_n\right) \right)_{\mathfrak{p}_{\lambda}} \cong A_{\lambda} \left( K_{\lambda} \right),
        $$
        where the isomorphism is an isomorphism of $K_\lambda$-algebras.
	\end{thrm}

	Lemma~\ref{lem:dimension} and Theorem~\ref{th:jet_scheme} imply the following corollary, which gives an alternative proof of the main result of \cite{Yeun}.

	\begin{cor}
		For every $\lambda = (\lambda_1, \ldots, \lambda_n)$ such that $\lambda_1 + \ldots + \lambda_n = m + 1$, the multiplicity of the ideal $J_m(x_1\cdots x_n)$ along $\mathfrak{p}_{\lambda}$ is equal to $\frac{(m + 1)!}{\lambda_1! \cdots \lambda_n!}$.
	\end{cor}

	\begin{proof}[Proof of Theorem~\ref{th:jet_scheme}]
    	For every $i \leqslant n$, we set $\Lambda_i = \lambda_1 + \ldots + \lambda_{i - 1}$ (and $\Lambda_0 = 0$).
        We will denote $J_m(x_1\cdots x_n)$ by $J_m$.
    
    	Assume that $\lambda_i = m + 1$ for some $i$, say for $i = 1$.
        Then $\lambda = (m + 1, 0, \ldots, 0)$.
        The ideal generated by $J_m$ in the localization $R_{\mathfrak{p}_{\lambda}}$ is $(x_1^{(0)}, \ldots, x_1^{(m)})$, so $\left( R / J_m\right)_{\mathfrak{p}_{\lambda}} \cong K_{\lambda}$.
        Since $A_{\lambda}\left( K_{\lambda} \right) \cong K_{\lambda}$, in what follows we can assume that $\lambda_i \leqslant m$ for all $i$.
        
        By $Q_{\lambda}$ we denote the primary component of $J_m$ corresponding to prime ideal $\mathfrak{p}_\lambda$.
        Let $B = \left( R / J_m\right)_{\mathfrak{p}_{\lambda}} = \left( R / Q_{\lambda}\right)_{\mathfrak{p}_{\lambda}}$.
        Then $B$ is a local $K_{\lambda}$-algebra with maximal ideal $\mathfrak{p}_{\lambda}B$.
        For every element $a \in R_{\mathfrak{p}_{\lambda}}$, by $\overline{a}$ we will denote its image in $B$.

		We also introduce a local algebra $C_{\lambda}(k)$ defined by 
        $$C_{\lambda}(k) = k\left[ y_{i, j} \mid 1\leqslant i \leqslant n, 0 \leqslant j < \lambda_i \right] / I,$$ 
        where $I$ is generated by nonleading coefficients of the polynomial
        $$
        \prod\limits_{i = 1}^{n} \left( y_{i, 0} + y_{i, 1}t + \ldots + y_{i, \lambda_i - 1} t^{\lambda_i - 1} + t^{\lambda_i} \right).
        $$
        We denote the coefficient of the monomial $t^k$ in the above polynomial by $f_k$.
        Thus, $I = (f_0, \ldots, f_{m})$.
        It will turn out that $A_{\lambda}(K_{\lambda})$, $B$, and $C_{\lambda}(K_{\lambda})$ are isomorphic.
        
        
        \begin{lem}\label{lem:AandC}
        	For every field $k$, local $k$-algebras $A_{\lambda}(k)$ and $C_{\lambda}(k)$ are isomorphic.
        \end{lem}
        
        \begin{proof}
        	Consider two polynomial algebras $S_1 = k\left[ y_{i, j} \mid 1\leqslant i \leqslant n, 0 \leqslant j < \lambda_i \right]$ and $S_2 = k[z_1, \ldots, z_{m + 1}]$.
            Then $C_{\lambda}(k) = S_1 / I$, where $I$ is defined above, and $A_{\lambda}(k) \subset A(k) = S_2 / I_S$, where $I_S$ is the ideal generated by all polynomials symmetric with respect to $z_1, \ldots, z_{m + 1}$ without constant terms.
            Furthermore, 
            \[
            I_S = (\sigma_1(z_1, \ldots, z_{m + 1}), \ldots, \sigma_{m + 1}(z_1, \ldots, z_{m + 1})),
            \]
            where $\sigma_d(z_1, \ldots, z_{m + 1})$ denotes the elementary symmetric polynomial in $z_1, \ldots, z_{m + 1}$ of degree $1 \leqslant d \leqslant m + 1$.
            We define a $k$-algebra homomorphism $\alpha \colon S_1 \to S_2$ by 
            \begin{equation}\label{eq:defining}
            \alpha(y_{i, j}) = \sigma_{\lambda_i - j}(z_{\Lambda_i + 1}, \ldots, z_{\Lambda_{i + 1}}).
            \end{equation}
            Since elementary symmetric polynomials are algebraically independent, $\alpha$ is injective.
            Extending $\alpha$ to $S_1[t] \to S_2[t]$ by $\alpha(t) = t$, we obtain
            \begin{align*}
            	&\alpha\left( \prod\limits_{i = 1}^{n} \left( y_{i, 0} + y_{i, 1}t + \ldots + y_{i, \lambda_i - 1} t^{\lambda_i - 1} + t^{\lambda_i} \right) \right) = \\
                &\prod\limits_{i = 1}^{n} \left( \alpha(y_{i, 0}) + \alpha(y_{i, 1})t + \ldots + \alpha(y_{i, \lambda_i - 1}) t^{\lambda_i - 1} + t^{\lambda_i} \right) = \\
                &\prod_{i = 1}^{m + 1} (t + z_i) = \sum\limits_{i = 0}^{m + 1} \sigma_{m + 1 - i}(z_1, \ldots, z_{m + 1}) t^i.
            \end{align*}
            Hence, $\alpha(f_k) = \sigma_{m + 1 - k}(z_1, \ldots, z_{m + 1})$ for $0 \leqslant k \leqslant m$, so $\alpha(I) \subset \alpha(S_1) \cap I_S$.
            We claim that $\alpha(I) = \alpha(S_1) \cap I_S$.
            Let $\alpha(p) \in \alpha(S_1) \cap I_S$, where $p \in S_1$.
            Since $\alpha(p)$ is $\lambda$-symmetric, Lemma~\ref{lem:ideal} implies that 
            there exist $\lambda$-symmetric $a_1, \ldots, a_{m + 1} \in S_2$ such that
            $\alpha(p) = \sum\limits_{j = 1}^{m + 1}a_j \sigma_j(z_1, \ldots, z_{m + 1})$.
            Lemma~\ref{lem:fundamental} implies that there exist $q_1, \ldots, q_{m + 1} \in S_1$ such that $a_j = \alpha(q_j)$ for all $1 \leqslant j \leqslant m + 1$.
            Hence, $\alpha(p) = \alpha\left( \sum\limits_{j = 1}^{m + 1} q_j f_{m + 1 - j} \right)$, so the injectivity of $\alpha$ implies that $p \in I$.
            
            Since $\alpha(I) = \alpha(S_1) \cap I_S$, $\alpha$ defines an injective homomorphism $\beta\colon S_1 / I \to S_2 / I_S$, which gives us the embedding $\beta\colon C_\lambda(k) \to A(k)$.
            Lemma~\ref{lem:fundamental} together with~\eqref{eq:defining} imply that the image $\beta(C_{\lambda}(k))$  coincides with the set of all $\lambda$-symmetric polynomials, so $\beta$ gives a desired isomorphism.
        \end{proof}
    
    
		\begin{lem}\label{lem:CandB}
			There exists a surjective homomorphism $\varphi_{\lambda}\colon C_{\lambda}(K_{\lambda}) \to B$.
		\end{lem}    
        
        \begin{proof}
	        We set $X_i(t) = x_i^{(0)} + x_i^{(1)}t + \ldots + x_i^{(m)}t^m \in R_{\mathfrak{p}_{\lambda}}[t] / \left( t^{m + 1} \right)$.
	        For fixed $i \leqslant n$, we will find $b_{0}, \ldots, b_{m - \lambda_i} \in R_{\mathfrak{p}_{\lambda}}$ such that the polynomial $X_i(t) P_i(t)$, where $P_i(t) = b_0 + b_1 t + \ldots + b_{m - \lambda_i} t^{m - \lambda_i}$, is equal to some monic polynomial of degree $\lambda_i$ modulo $t^{m + 1}$.
        	Then elements $b_{0}, \ldots, b_{m - \lambda_i}$ should satisfy the following linear system
	        $$
	        \begin{pmatrix}
				x_i^{(\lambda_i)} & \ldots &  x_i^{(m - 1)}          & x_i^{(m)}               \\
	            \vdots            & \ddots &  \vdots                 & \vdots                  \\
	            \ldots            & \ldots & x_{i}^{(\lambda_i)}     & x_{i}^{(\lambda_i + 1)} \\
	            \ldots            & \ldots & x_{i}^{(\lambda_i - 1)} & x_{i}^{(\lambda_i)} 
			\end{pmatrix}
	        \begin{pmatrix}
				b_{m - \lambda_i} \\
	            \vdots \\
	            b_1 \\
	            b_0
			\end{pmatrix}
	        =
	        \begin{pmatrix}
				0 \\
	            \vdots \\
	            0 \\
	            1
			\end{pmatrix}.
	        $$
            
        	The matrix $M$ of this system is upper-triangular modulo ideal $\mathfrak{p}_\lambda R_{\mathfrak{p}_{\lambda}}$ with $x_i^{(\lambda_i)}$ on the diagonal, so there is a unique vector $(b_{m - \lambda_i}, \ldots, b_1, b_0) \in R_{\mathfrak{p}_{\lambda}}^{m - \lambda_i + 1}$ satisfying this system.
            Moreover, $b_0 \notin \mathfrak{p}_\lambda R_{\mathfrak{p}_{\lambda}}$.
        
	        We denote the product $X_i(t) P_i(t)$ reduced modulo $t^{m + 1}$ by $Q_i(t) = a_{i, 0} + a_{i, 1}t + \ldots + a_{i, \lambda_i - 1} t^{\lambda_i - 1} + t^{\lambda_i}$.
            Since $x_i^{(0)}, \ldots, x_i^{(\lambda_i - 1)} \in \mathfrak{p}_\lambda R_{\mathfrak{p}_{\lambda}}$, elements $a_{i, 0}, \ldots, a_{i, \lambda_i - 1}$ also lie in $\mathfrak{p}_\lambda R_{\mathfrak{p}_{\lambda}}$, so their images in $B$ are nilpotents.
            We also note that
            \begin{equation}\label{eq:productQ}
            \prod\limits_{i = 1}^n Q_i(t) \equiv \left( \prod\limits_{i = 1}^n X_i(t) \right) \cdot \left( \prod\limits_{i = 1}^n P_i(t)\right) \pmod{t^{m + 1}}.
            \end{equation}

    In what follows, the image of element $a \in R_{\mathfrak{p}_\lambda}$ in $B$ will be denoted by $\overline{a}$.
    By $S$ we denote the $k$-subalgebra of $B$ generated by $\overline{a}_0, \ldots, \overline{a}_{\lambda_i - 1}, \overline{x}_i^{(\lambda_i)}, \ldots, \overline{x}_{i}^{(m)}$, and the inverse of $\overline{x}_{i}^{(\lambda_i)}$, we set also $\mathfrak{p} = \mathfrak{p}_{\lambda} B \cap S$.
    We claim that $\overline{x}_i^{(0)}, \ldots, \overline{x}_i^{(\lambda_i - 1)} \in S$.
    For every~$j$, $\overline{b}_j$ can be written as $\frac{c_j}{c}$, where $c = \det \overline{M}$ and $c_j$ is a polynomial in $\overline{x}_i^{(\lambda_i)}, \ldots, \overline{x}_i^{(m)}$ over $k$.
        	Then the following matrix equality
	        \begin{equation}\label{eq:hensel}
	        \begin{pmatrix}
				c_0    & c_1    & \ldots & c_{\lambda_i - 1} \\
	            0      & c_0    & \ldots & c_{\lambda_i - 1} \\
	            \vdots & \ddots & \ddots & \vdots            \\
	            0      & \ldots & 0      & c_0
			\end{pmatrix}
	        \begin{pmatrix}
				\overline{x}_i^{(\lambda_i - 1)} \\
	            \overline{x}_i^{(\lambda_i - 2)} \\
	            \vdots \\
	            \overline{x}_i^{(0)}
			\end{pmatrix}
	        =
	        \begin{pmatrix}
				c \overline{a}_{i, \lambda_i - 1} \\
	            c \overline{a}_{i, \lambda_i - 2} \\
	            \vdots \\
	            c \overline{a}_{i, 0}
			\end{pmatrix}
	        \end{equation}
	        can be considered as a system of polynomial equations in $\overline{x}_i^{(0)}, \ldots, \overline{x}_{i}^{(\lambda_i - 1)}$ over $S$.
	        We denote the matrix in the left-hand side of~\eqref{eq:hensel} by $N$.
        	Vector $(0, \ldots, 0) \in S^{\lambda_i}$ gives a solution of this system modulo $\mathfrak{p}$.
	        The jacobian matrix of the system~\eqref{eq:hensel} at $(0, \ldots, 0)$ is equal to $N$ modulo~$\mathfrak{p}$.
	        The determinant $\det N = c_0^{\lambda_i}$ at $(0, \ldots, 0)$ is equal to $\left( x_i^{(\lambda_i)} \right)^{\lambda_i(m - \lambda_i)}$ modulo $\mathfrak{p}$, so it is an invertible element of $S$.
            Since $\mathfrak{p}$ is a nilpotent ideal, $S$ is complete with respect to it.
	        Thus, Hensel's Lemma (\cite[Sect. 4]{Kuhlmann}) implies that there exists a unique solution of~\eqref{eq:hensel} in $S^{\lambda_i}$, which is equal to $(0, \ldots, 0)$ modulo $\mathfrak{p}$.
	        Repeating the same argument for algebra $B$ and ideal $\mathfrak{p}_{\lambda}B$, we have that~\eqref{eq:hensel} has a unique solution in $B \supset S$.
	        Hence, these solutions coincide, and are both equal to $\left( \overline{x}_i^{(0)}, \ldots, \overline{x}_{i}^{(\lambda_i - 1)} \right)$.
	        So, $\overline{x}_i^{(0)}, \ldots, \overline{x}_{i}^{(\lambda_i - 1)} \in S$, and $B$ is generated by $\{ a_{i, j} \mid 1 \leqslant i \leqslant n, 0 \leqslant j < \lambda_i \}$ as a $K_\lambda$-algebra.

            Consider $\varphi\colon K_{\lambda}\left[ y_{i, j} \mid 1 \leqslant i \leqslant n, 0 \leqslant j < \lambda_i \right] \to B$ defined by $\varphi(y_{i, j}) = \overline{a}_{i, j}$.
            Since the coefficiets of $\prod_i^n X_i(t)$ generate $J_m(x_1\cdots x_n)$, equality~\eqref{eq:productQ} implies that $\Ker\varphi \supset I$, so we obtain a homomorphism $\varphi_{\lambda}\colon C_{\lambda}(K_{\lambda}) \to B$.
            Since $a_{i, j}$ generate $B$ over $K_{\lambda}$, $\varphi_{\lambda}$ is surjective, so we are done.
        \end{proof}
                
        Let us return to the proof of Theorem~\ref{th:jet_scheme}.
        Lemma~\ref{lem:AandC} and Lemma~\ref{lem:CandB} imply that for every $\lambda \in \mathbb{Z}_{\geqslant 0}^n$ such that $\lambda_1 + \ldots + \lambda_n = m + 1$ inequality 
        \begin{equation}\label{eq:dimension}
        \dim_{K_{\lambda}} \left( R / J_m\left( x_1 \cdots x_n\right) \right)_{\mathfrak{p}_{\lambda}} \leqslant \dim_{K_{\lambda}} A_{\lambda}(K_{\lambda})
        \end{equation}
        holds.
        Moreover, the equality in~\eqref{eq:dimension} would imply that $\varphi_{\lambda}$ is an isomorphism.
        
        The left-hand side of~\eqref{eq:dimension} is the multiplicity of $J_m$ along $\mathfrak{p}_{\lambda}$.
        Every $\mathfrak{p}_{\lambda}$ defines an affine subspace, so has degree one.
        Since $J_m$ is an ideal of codimension $m + 1$ generated by $m + 1$ polynomials, B\'ezout theorem (\cite[Th. 1.7.7]{Hartshorne}) implies that the sum of multiplicities of these prime components is equal to the product of degrees of generating polynomials. 
        The latter is equal to $n^{m + 1}$.
        Summing up inequality~\eqref{eq:dimension} for all $\lambda$ such that $\lambda_1 + \ldots + \lambda_n = m + 1$, we obtain
        $$
        n^{m + 1} \leqslant \sum\limits_{\lambda} \dim A_{\lambda}(K_{\lambda}) = (\text{Lemma~\ref{lem:dimension}}) = \sum\limits_{\lambda_1 + \ldots + \lambda_n = m + 1} \frac{(m + 1)!}{\lambda_1!\cdots \lambda_n!} = n^{m + 1}.
        $$
        Thus, all inequalities of the form~\eqref{eq:dimension} are equalities, so all $\varphi_{\lambda}$ are isomorphisms.
	\end{proof}
    
    \begin{cor}\label{cor:image}
		Let us consider the isomorphism 
        $$
        \psi_{\lambda} \colon \left( R / J_m\left( x_1 \cdots x_n\right) \right)_{\mathfrak{p}_{\lambda}} \to A_{\lambda} \left( K_{\lambda} \right)
        $$ 
        from Theorem~\ref{th:jet_scheme}.
        For every $1 \leqslant i \leqslant n$ and $0 \leqslant j < \lambda_i$, the image $\psi_{\lambda}\left( x_i^{(j)} \right)$ lies in $A_{\lambda}^i \left( K_{\lambda} \right)$.
        
        Moreover, $\nu\left( \psi_{\lambda}\left( x_i^{(j)} \right) \right) = \lambda_i - j$ (see Definition~\ref{def:main}).
	\end{cor}
    
    \begin{proof}
    	In what follows we will use the notation from the proof of Theorem~\ref{th:jet_scheme}.
    	Let us fix $i$ and $j$ such that $1 \leqslant i \leqslant n$ and $0 \leqslant j < \lambda_i$.
		The proof of Lemma~\ref{lem:CandB} implies that the image of $x_{i}^{(j)}$ in the algebra $C_{\lambda}(K_{\lambda})$ lies in the subalgebra generated by $y_{i, 0}, \ldots, y_{i, \lambda_i - 1}$.
        The construction of the isomorphism between $C_{\lambda}(K_{\lambda})$ and $A_{\lambda}(K_{\lambda})$ (see~\eqref{eq:defining}) implies that the images of $y_{i, 0}, \ldots, y_{i, \lambda_i - 1}$ lie in $A_{\lambda}^i \left( K_{\lambda} \right)$.
    
    	In order to prove the second claim we need to introduce some notation.
    	Due to the proof of Theorem~\ref{th:jet_scheme}, the isomorphism $\psi_{\lambda}$ can be factored as $\psi_{\lambda} = \beta \circ \gamma$, where $\beta\colon C_{\lambda}(K_{\lambda}) \to A_{\lambda}(K_{\lambda})$ is the isomorphism constructed in Lemma~\ref{lem:AandC}, and 
        \[
        \gamma\colon B_{\lambda} = \left( R / J_m\left( x_1 \cdots x_n\right) \right)_{\mathfrak{p}_{\lambda}} \to C_{\lambda}(K_{\lambda})
        \]
        is the inverse to the isomorphism $\varphi_{\lambda}$ constructed in Lemma~\ref{lem:CandB} (its injectivity proved in Theorem~\ref{th:jet_scheme}).
        For elements $b \in B_{\lambda}$ and $c \in C_{\lambda}(K_{\lambda})$, we can define function $\nu$ by $\nu(c) = \nu(\beta(c))$ and $\nu(b) = \nu(\gamma(b))$.
        Our goal can be reformulated as $\nu(x_{i}^{(j)}) \geqslant \lambda_i - j$.
        Definitions of $\beta$ and $\gamma$ imply that $\nu(\overline{a}_{i, j}) = \nu(y_{i, j}) = \lambda_i - j$, and for every invertible element $z$ from $A_{\lambda}(K_{\lambda})$, $B_{\lambda}$, or $C_{\lambda}(K_{\lambda})$ we have $\nu(z) = 0$.
        
        We will prove the second claim of the corollary by induction on $j$.
        For the case $j = 0$, we consider the last equation in~\eqref{eq:hensel}.
        Since $c$ and $c_0$ are invertible, we have $\nu(\overline{x}_i^{(0)}) = \nu(\overline{a}_{i, 0}) = \lambda_i$.
        For an arbitrary $j$, we consider $(\lambda_i - j)$-th equation in~\eqref{eq:hensel}.
        Again, due to invertibility of $c$, $\nu$ of the right-hand side is $\lambda_i - j$.
        The value of $\nu$ of all summands on the left-hand side except $c_0 \overline{x}_{i}^{(j)}$ is larger than $\lambda_i - j$ due to the induction hypothesis.
        Since $c_0$ is invertible, $\nu\left( \overline{x}_i^{(j)} \right) = \lambda_i - j$.
    \end{proof}


\section{Applications to differential algebra}

\begin{thrm}\label{th:degree}
	Let $K$ be a differential field, $(h_1, \ldots, h_n) \in \mathbb{Z}_{\geqslant 0}^n$, and $H = h_1 + \ldots + h_n$.
    By $I_h$ we denote the ideal 
    $$
    \left( y_1\cdots y_n, (y_1\cdots y_n)^{\prime}, \ldots, (y_1\cdots y_n)^{(h)} \right) \subset K\{ y_1, \ldots y_n \}.
    $$
    Then
    \begin{enumerate}
		\item The differential polynomial $y_1^{(h_1)} \cdots y_n^{(h_n)}$ does not lie in the radical of $I_{H - 1}$.
        \item The minimal $d$ such that $\left( y_1^{(h_1)} \cdots y_n^{(h_n)} \right)^d \in I_H$ is equal to $\max\limits_{1\leqslant i \leqslant n} (h_i + 1)(H - h_i) + 1$.
	\end{enumerate}
\end{thrm}

\begin{proof}
	The first claim can be deduced from the result of \cite{GowardSmith}, but we prefer to give a more straightforward argument.
    Let us consider a (not differential) homomorphism $\varphi\colon K\{ y_1, \ldots, y_n\} \to K$ defined by 
    $$
    \varphi\left( y_i^{(j)} \right) = 
    \begin{cases}
		0, \mbox{ if } j < h_i, \\
        1, \mbox{ otherwise.}
	\end{cases}
    $$
    For every tuple $(i_1, \ldots, i_n) \in \mathbb{Z}_{\geqslant 0}^n$ such that $i_1 + \ldots + i_n = h \leqslant H - 1$, there exists $j$ such that $i_j < h_j$.
    Hence, $\varphi\left(y_1^{(i_1)} \cdots y_n^{(i_n)}\right) = 0$, so $\varphi\left( \left(y_1\cdots y_n\right)^{(h)} \right) = 0$.
    On the other hand, $\varphi\left( y_1^{(h_1)} \cdots y_n^{(h_n)} \right) = 1$.
    Thus, the first claim is proved.
    
    By $D$ we denote the number $\max\limits_{1\leqslant i \leqslant n} (h_i + 1)(H - h_i)$.
    We want to formulate our problem using the language of jet ideals.
    Since the ideal $\widetilde{I}_H = I_H \cap K\left[ y_i^{(j)} \mid 1 \leqslant i \leqslant n, j \leqslant H\right]$ is also generated by $y_1\cdots y_n, (y_1\cdots y_n)^{\prime}, \ldots, (y_1\cdots y_n)^{(H)}$, it sufficient to consider $\widetilde{I}_H$.
    Similarly to Remark~\ref{rem:jet_to_diff}, the isomorphism 
    $$
    \psi\colon K\left[ y_i^{(j)} \mid 1 \leqslant i \leqslant n, j \leqslant H\right] \to R_H = K\left[ x_i^{(j)} \mid 1 \leqslant i \leqslant n, j \leqslant H\right]
    $$
    defined by $\psi(y_i^{(j)}) = j! x_i^{(j)}$ sends $\widetilde{I}_H$ to the $H$-th jet ideal of $x_1 \cdots x_n$.
    We denote the latter by $J_H$.
    Since $\psi\left( y_1^{(h_1)} \cdots y_n^{(h_n)} \right) = \left( \prod_{i = 1}^n h_i! \right) x_1^{(h_1)} \cdots x_n^{(h_n)}$, 
    it is sufficient to prove that the minimal $d$ such that $\left( x_1^{(h_1)} \cdots x_n^{(h_n)} \right)^d \in J_H$ is equal to $D + 1$.
    
    First, we prove that for every $d > D$
    $$
    \left( x_1^{(h_1)} \cdots x_n^{(h_n)} \right)^d \in J_H.
    $$
    Let us consider a minimal prime of $J_H$.
    Due to the discussion before Theorem~\ref{th:jet_scheme}, it is of the form $\mathfrak{p}_{\lambda}$, where $\lambda = (\lambda_1, \ldots, \lambda_n) \in \mathbb{Z}_{\geqslant 0}^n$ such that $\lambda_1 + \ldots + \lambda_n = H + 1$.
    Theorem~\ref{th:jet_scheme} also implies that 
    $$
    B_{\lambda} = \left( K\{x_1, \ldots, x_n\} / J_H \right)_{\mathfrak{p}_{\lambda}} \cong A_{\lambda} \left( L \right)
    $$
    for some field $L$.
    It is sufficient to show that for all $\lambda$ the image of $\left( x_1^{(h_1)} \cdots x_n^{(h_n)} \right)^d$ in this localization is zero.
    Since $\lambda_1 + \ldots + \lambda_n = H + 1 > h_1 + \ldots + h_n$, there exists $i$ such that $h_i < \lambda_i$.
    Due to Corollary~\ref{cor:image}, $\nu\left( \psi_{\lambda}\left( x_i^{(h_i)} \right) \right) = \lambda_i - h_i$, so Proposition~\ref{prop:nilpotent_zero} implies, that in~$B_{\lambda}$
    $$
    \left( x_i^{(h_i)} \right)^{d_0} = 0, \mbox{ if } d_0 > \frac{\lambda_i}{\lambda_i - h_i}(H + 1 - \lambda_i).
    $$
    Since $H - h_i \geqslant H + 1 - \lambda_i$ and $(\lambda_i - h_i)(h_i + 1) \geqslant \lambda_i$, so we obtain 
    \[
    D \geqslant (h_i + 1)(H - h_i) \geqslant \frac{\lambda_i}{\lambda_i - h_i}(H + 1 - \lambda_i).
    \]
    Hence, $\left( x_i^{(h_i)} \right)^d = 0$ in $B_{\lambda}$.
    
    Now it is sufficient to prove that $\left( x_1^{(h_1)} \cdots x_n^{(h_n)} \right)^D \notin J_H$.
    Without loss of generality, we can assume that $h_1 \geqslant h_2 \geqslant \ldots \geqslant h_n$, so $D = (h_1 + 1)(H - h_1)$.
    Set $\lambda = (h_1 + 1, h_2, \ldots, h_n)$ and 
    \[
    H_i = (h_1 + 1) + h_2 + \ldots + h_{i - 1} = \lambda_1 + \ldots + \lambda_{i - i}
    \]
    for all $i \leqslant n$.
    We define a homomorphism $\psi\colon R_H \to A_{\lambda}(K) \subset K[z_1, \ldots, z_{H + 1}] / I_S$ by
    $$
    \psi\left( x_i^{(j)} \right) = 
    \begin{cases}
		\sigma_{\lambda_i - j} \left( z_{H_i + 1}, \ldots, z_{H_{i + 1}}\right), \mbox{ if } j < \lambda_i, \\
        1, \mbox{ if } j = \lambda_i,\\
        0, \mbox{ otherwise.}
	\end{cases}
    $$
    We claim that $J_H \subset \Ker\psi$.
    Note that coefficients of $1, t, \ldots, t^H$ in the product
    $$
    \prod\limits_{i = 1}^n \left( \psi\left(x_i^{(0)}\right) + \psi\left(x_i^{(1)}\right) t + \ldots + \psi\left(x_i^{(H)}\right) t^H \right) \in A_{\lambda}(K) [t]
    $$
    are exactly images of generators of $J_H$.
    Similarly to the proof of Lemma~\ref{lem:AandC}, we obtain
    \begin{multline*}
	\prod\limits_{i = 1}^n \left( \psi\left(x_i^{(0)}\right) + \psi\left(x_i^{(1)}\right) t + \ldots + \psi\left(x_i^{(H)}\right) t^H \right) = \prod\limits_{i = 1}^{H + 1} (t + z_i) = \\ \sum\limits_{i = 1}^{H + 1} \sigma_{i}(z_1, \ldots, z_{H + 1}) t^{H + 1 - i} + t^{H + 1} = t^{H + 1}.
	\end{multline*}
    Hence, $J_H \subset \Ker\psi$.
    On the other hand, $\psi\left( x_1^{(h_1)}  \ldots x_n^{(h_n)}\right) = z_1 + \ldots + z_{h_1 + 1}$.
    Proposition~\ref{prop:nilpotent_nonzero} implies that $\psi\left( \left( x_1^{(h_1)}  \cdots x_n^{(h_n)} \right)^D \right) \neq 0$, so $\left( x_1^{(h_1)}  \ldots x_n^{(h_n)} \right)^D \notin J_H$.
\end{proof}

\begin{cor}\label{cor:product_of_ideals}
	Let $n$ be a positive integer, and $h_1, \ldots, h_n$ be non negative integers with the sum $H := h_1 + \ldots + h_n$.
    Then for ideals (not necessarily differential) $I_1, \ldots, I_n$ in a differential $K$-algebra $A$ the following inclusion holds
    $$
    \left( I_1^{(h_1)} \cdots I_n^{(h_n)} \right)^{\max(h_i + 1)(H - h_i) + 1} \subset \left( I_1\cdots I_n \right)^{(h_1 + \ldots + h_n)}.
    $$
\end{cor}

\begin{proof}
	In Section~5 of \cite{OvchinnikovPogudinVo} (see Corollary 6.2 and Lemma 6.3) it was proved that if for some~$s$
    $$
    \left( x_1^{(h_1)} \cdots x_n^{(h_n)} \right)^s \in \left( x_1\cdots x_n, (x_1\cdots x_n)^{\prime}, \ldots, (x_1\cdots x_n)^{(h_1 + \ldots + h_n)}  \right) \subset K\{ x_1, \ldots, x_n \},
    $$
    then for the same $s$ the inclusion
    \begin{equation}\label{eq:inclusion_polys}
    \left( I_1^{(h_1)} \cdots I_n^{(h_n)} \right)^{s} \subset \left( I_1\cdots I_n \right)^{(h_1 + \ldots + h_n)}
    \end{equation}
    holds for any ideals in any differential polynomial ring $K\{y_1, \ldots, y_N\}$.
    
    We claim that the same inclusion also holds in any commutative associative differential $K$-algebra $A$.
    Let $I_1, \ldots, I_n$ be ideals in $A$.
    Let $a$ be any element of $\left( I_1^{(h_1)} \cdots I_n^{(h_n)} \right)^{s}$. 
    Then there are finitely generated ideals $J_1 \subset I_1$, $\ldots$, $J_n \subset I_n$ such that $a \in \left( J_1^{(h_1)} \cdots J_n^{(h_n)} \right)^{s}$.
    Let $a_1, \ldots, a_N$ be the union of  the sets of generators of $J_1, \ldots, J_n$.
    Consider a homomorphism of differential $K$-algebras $\varphi\colon K\{y_1, \ldots, y_N\} \to A$ defined by $\varphi(y_i) = a_i$.
    Then~\eqref{eq:inclusion_polys} implies that every $b \in \varphi^{-1}(a)$ belongs to $\left( \varphi^{-1}(J_1)\cdots \varphi^{-1}(J_n) \right)^{(h_1 + \ldots + h_m)}$.
    Applying $\varphi$, we conclude that \[a \in \left( J_1 \cdots J_n \right)^{(h_1 + \ldots + h_m)} \subset \left( I_1 \cdots I_n \right)^{(h_1 + \ldots + h_m)}, \] so the claim is proved.
    
    Now the corollary follows directly from Theorem~\ref{th:degree}.
\end{proof}

\begin{remark}
	Theorem~\ref{th:degree} also shows that the degree in Corollary~\ref{cor:product_of_ideals} can not be lower.
    As an example we can take principal ideals $I_1 = (x_1)$, $I_2 = (x_2)$, $\ldots$, $I_n = (x_n)$ in $K\{ x_1, \ldots, x_n\}$.
\end{remark}

Analogously to~\cite[Corollary~6.3]{OvchinnikovPogudinVo} one can use the inclusions
\[
\left( J_1 \cap \ldots \cap J_n \right)^n \subset J_1 \cdots J_n \subset J_1 \cap \ldots \cap J_n
\]
valid for arbitrary ideals in any $K$-algebra
to deduce the following result from Corollary~\ref{cor:product_of_ideals}.

\begin{cor}\label{cor:intersection}
	Let $n$ be a positive integer, and $h_1, \ldots, h_n$ be non negative integers with the sum $H := h_1 + \ldots + h_n$.
    Then for ideals (not necessarily differential) $I_1, \ldots, I_n$ in a differential $K$-algebra $A$, the following inclusion holds
    $$
    \left( I_1^{(h_1)} \cap \ldots \cap I_n^{(h_n)} \right)^{n \max(h_i + 1)(H - h_i) + n} \subset \left( I_1\cap \ldots \cap I_n \right)^{(h_1 + \ldots + h_n)}.
    $$
\end{cor}

\begin{remark}\label{rem:levi}
	Theorem~\ref{th:degree} also gives a partial solution for classical membership problem for differential ideal $[x_1\ldots x_n] \in K\{x_1, \ldots, x_n\}$.
    More precisely, it implies that if for every $h_1, \ldots, h_n \in \mathbb{Z}_{\geqslant 0}$ and $H = h_1 + \ldots + h_n$
    \begin{equation}\label{eq:diff_membership}
	    \left( x_1^{(h_1)}\cdots x_n^{(h_n)} \right)^d \in [x_1 \cdots x_n],
	\end{equation}
    where $d = \max\limits_{1\leqslant i \leqslant n} (h_i + 1)(H - h_i) + 1$.
    Possibly, the result of such form can be also deduced using techniques of \cite{HillmanEtAl2}, but it is not straightforward.
    
    Let us also note, that unlike Theorem~\ref{th:degree} the degree $d$ in~\eqref{eq:diff_membership} is not optimal in general.
    In other words, taking more prolongations, it is possible to lower the degree $d$.
    For example, in the case $n = 2$ and $h_1 = h_2 = 2$ the result of Levi (\cite[Th. 3.2]{Levi}) implies that $\left( x_1^{(2)} x_2^{(2)}\right)^5 \in [x_1x_2]$.
\end{remark}

\section*{Acknowledgements}

This work was supported by the Austrian Science Fund FWF grant Y464-N18 and NSF grants CCF-1563942, CCF-0952591, DMS-1606334.

The author is grateful to Alexey Ovchinnikov and Thieu Vo Ngoc for useful discussions and to the referee for numerous comments and suggestions, which helped to improve the paper.


\setlength{\bibsep}{2pt}
\bibliographystyle{abbrvnat}
\bibliography{bibdata}

\end{document}